\documentclass[12pt,reqno]{amsart}

\usepackage{amsmath, amssymb, amsthm, mathtools}
\usepackage{graphicx, hyperref}
\usepackage[msc-links,initials]{amsrefs}
\usepackage[inner=1.33in,outer=1.33in,bmargin=1.5in]{geometry}

\hypersetup{
    colorlinks=true,
    linkcolor=blue,
    citecolor=blue,
    urlcolor=blue,
    }

\newtheorem{theorem}{Theorem}

\newtheorem{lemma}{Lemma}
\newtheorem{proposition}{Proposition}
\newtheorem{corollary}{Corollary}
\theoremstyle{remark}
\newtheorem{remark}{Remark}

\numberwithin{equation}{section}

\title[Explicit Gaps Between Squarefree Integers]{Explicit Bounds for Large Gaps Between Squarefree Integers}
\author[A. Kumchev]{Angel Kumchev}
\address{Department of Mathematics\\ Towson University\\ Towson, MD 21252\\ U.S.A.}
\email{akumchev@towson.edu}
\author[W. McCormick]{Wade McCormick}
\address{Department of Mathematics\\ University of California Berkeley\\ Berkeley, CA 94720\\ U.S.A.}
\email{wademc@berkeley.edu}
\author[N. McNew]{Nathan McNew}
\address{Department of Mathematics\\ Towson University\\ Towson, MD 21252\\ U.S.A.}
\email{nmcnew@towson.edu}
\author[A. Park]{Ariana Park}
\address{Department of Mathematics\\ Massachusetts Institute of Technology\\ Cambridge, MA 02139\\ U.S.A.}
\email{arianap@mit.edu}
\author[R. Scherr]{Russell Scherr}
\address{Department of Mathematics\\ Towson University\\ Towson, MD 21252\\ U.S.A.}
\email{rscher4@students.towson.edu}
\author[W. Ziehr]{Willow Ziehr}
\address{Department of Mathematics\\ University of California Berkeley\\ Berkeley, CA 94720\\ U.S.A.}
\email{simon\_ziehr@berkeley.edu}

\begin{document}

\begin{abstract}  We obtain explicit forms of the current best known asymptotic upper bounds for gaps between squarefree integers.  In particular we show, for any $x \ge 2$, that every interval of the form $(x, x + 11x^{1/5}\log x]$ contains a squarefree integer. The constant 11 can be improved further, if $x$ is assumed to be larger than a (very) large constant.
\end{abstract}

\maketitle
\section{Introduction}

An integer $n$ is called {\em squarefree} if it is not divisible by the square of any prime $p$. More generally, if $k \ge 2$, $n$ is called {\em $k$-free} if it is not divisible by $p^k$ for any prime $p$; $3$-free integers, in particular, are also known as \emph{cubefree}.

The asymptotic distribution of the $k$-free integers has been studied systematically, at least since the early 1900s, with a special focus on the squarefree case. Let $Q_k(x)$ denote the counting function of the $k$-free numbers up to $x$, and consider the error term $E_k(x)$ in the asymptotic formula
\[ Q_k(x) = \frac {x}{\zeta(k)} + E_k(x), \]
where $\zeta(k)$ is the Riemann zeta-function. The bound $E_k(x) = O\left(x^{1/k}\right)$ is classical, and further improvements are closely related to the distribution of zeros of the zeta-function. In particular, the best known bound for $E_k(x)$,
\[ E_k(x) = O \left( x^{1/k}\exp\left(-c(k)(\log x)^{3/5}(\log\log x)^{-1/5}\right) \right), \]
follows from the work of Walfisz on the error term in the Prime Number Theorem (see \cite{Walf63}). Still, a number of authors \cite{BaPi85, BaPo10, Grah81, Jia93, HLiu14, HLiu16, MoVa81} have obtained sharper bounds under the assumption that the Riemann Hypothesis is true.

A related problem that has attracted considerable attention concerns the gaps between consecutive $k$-free integers. The first result in that direction was obtained by Fogels~\cite{fogel41}, who proved that if $\theta > 2/5$ the interval $( x,x+x^{\theta} ]$ contains a squarefree integer for all sufficiently large $x$. In 1951, Roth \cite{Roth51} reduced the exponent $2/5$ in Fogels's result to $3/13$, while Halberstam and Roth \cite{HR51} proved that the interval $(x, x+x^{\theta}]$ contains a $k$-free integer for any $\theta > 1/(2k)$ and for all sufficiently large $x$. Around the same time, Erd\H os~\cite{erdos51} proved that there exist infinitely many intervals $(x,x+h]$, with
\[ h \gg \frac {\log x}{\log\log x}, \]
which contain no squarefree integers. Together, these results inspired the conjecture that for any fixed $\varepsilon > 0$, the interval $(x, x+x^\varepsilon]$ contains a squarefree integer for sufficiently large $x$. This conjecture seems beyond the reach of current methods, though Granville \cite{Gran98} has shown that, like many other famous theorems and conjectures in number theory, it follows from the $abc$-conjecture of Masser and Oesterl\'e.

Initially, further improvements on Roth's result \cite{Roth51} on gaps between squarefree numbers were obtained through the method of exponential sums \cite{gk88, rank55, richert54, Schm64}, while the (mostly elementary) work of Halberstam and Roth \cite{HR51} inspired research on the distribution of $k$-free numbers in polynomial sequences: see \cite{Hool67, HN80, nair79} for some early work and \cite[\S2]{FGT15} for a more detailed history. Starting in the late 1980s, Filaseta and Trifonov published a series of papers~\cite{Fila88, Fila90, FT89, FT92, FT96, tr89, Tr95}, where they developed an elementary proof \cite{FT92} that there exists a constant $c > 0$ such that the interval $(x, x+ cx^{1/5}\log x]$ contains a squarefree integer for all sufficiently large $x$. Later, Trifonov \cite{Tr95} generalized this result and proved that, for each $k \ge 3$, there exists a constant $c = c(k) > 0$ such that the interval $(x, x+ cx^{1/(2k+1)}\log x]$ contains a $k$-free integer for all sufficiently large $x$. Filaseta and Trifonov~\cite{FT96} generalized their method to achieve progress in other problems---see the survey article~\cite{FGT15} for the history of such developments, but sharper bounds on the gaps between $k$-free integers have remained elusive.

During the past couple of decades, number theorists' interest in numerically explicit results has increased significantly, and this has led to the development of numerically explicit versions of known theorems. As the Filaseta--Trifonov approach to gaps between $k$-free integers is both self-contained and ``numerically friendly,'' it therefore makes sense to investigate fully explicit versions of the results of \cite{FT92} and \cite{Tr95}. In this note, we prove such explicit versions of the gap results for squarefree integers. Our main theorem is as follows.

\begin{theorem}\label{thm2free1}
  For any $x \ge 2$, the interval $(x, x + 11x^{1/5}\log x]$ contains a squarefree integer.
\end{theorem}

The reader familiar with the work of Filaseta and Trifonov may wonder whether the techniques from this work can be extended to obtain similar results on gaps between $k$-free integers when $k \ge 3$. This is very much possible. Indeed, we have proved the following result on gaps between cubefree integers. Its proof and the proofs of companion results on $k$-free integers with $k \ge 4$ will appear in a forthcoming paper.

\begin{theorem}\label{thm3free1}
  For any $x \ge 2$, the interval $(x, x + 5x^{1/7}\log x]$ contains a cubefree integer.
\end{theorem}

The focus of the above theorems is on providing explicit intervals that work for all $x$. The price we pay for this universality are the somewhat elevated values of the constants $11$ and $5$ in the theorems. If one is interested in reducing those constants further and willing to accept a result that holds only for sufficiently large $x$, then one may prefer versions like those given in the next theorem.

\begin{theorem}\label{thm2free2}
Every interval 
\begin{itemize}
\item $(x,x+5x^{1/5}\log x]$ contains a squarefree number for  $x \ge e^{400}$; 
\item $(x,x+2x^{1/5}\log x]$ contains a squarefree number for  $x \ge e^{1800}$; 
\item $(x,x+x^{1/5}\log x]$ contains a squarefree number for  $x \ge e^{500\,000}$.
\end{itemize}
\end{theorem}

Mossinghoff, Oliveira e Silva and Trudgian \cite{MOT21} (see also Marmet \cite{marm12}) investigated long gaps between squarefree numbers numerically.  Their computational work establishes the size of the longest gaps up to $10^{18}$, which are all dramatically smaller than the bounds that we get in this paper.  The largest gap that they find is a string of 18 consecutive non-squarefree numbers, the first of which is $125\,781\,000\,834\,058\,568$.  As a result of their work, we can assume $x\geq 10^{18}> e^{41}$ throughout the rest of this paper.

Theorem \ref{thm2free2} already hints that the constants in Theorems \ref{thm2free1} and \ref{thm3free1} are influenced by the ``small'' values of $x$. Indeed, we establish Theorem \ref{thm2free1} for $x \ge e^{116}$. To bridge the gap between this lower bound and $e^{41}$, we prove several propositions giving results with larger exponents, which are however superior to the results of the main theorem for small $x$. In particular, we find that the interval $(x,x+5x^{1/4}]$ always contains a squarefree integer (Proposition~\ref{prop3}) and the interval $(x,x+3.8x^{1/4}]$ contains a squarefree integer for $x\ge e^{109}$ (Proposition~\ref{prop4}). 

It should be clear by now from the above discussion, that the values of the constants and the various cutoffs in the theorems (and in Propositions~\ref{prop3} and \ref{prop4}) are not exact, but rather ``nice'' approximations. We say more about this in Section~\ref{sec:8}.\footnote{The interested reader can explore these phenomena further using the {\tt SageMath} code for the computational part of our work, which is available at \url{https://github.com/agreatnate/explicit-k-free-integer-bounds}}

\subsection*{Notation} Throughout the paper, for a real number $\theta$, we use $\lfloor \theta \rfloor$ to denote the greatest integer less than or equal to~$\theta$; also, $\{ \theta \} = \theta - \lfloor \theta \rfloor$. We write $|A|$ for the size of the set $A$, and $\pi(x)$ for the prime counting function. 

\section{Preliminaries}
\label{sec:2}

\subsection{Outline of the method}

Let $N(x,h)$ be the number of integers in $(x,x+h]$ that are not squarefree. Clearly, to prove any of our theorems, it suffices to show that $N(x,h)<h-1$ for the respective choices of $x$ and $h$. We first sieve this interval of the squares of very small primes, up to a parameter $J$ to be chosen later.  The number of integers in $(x,x+h]$ divisible by the square of a prime up to $J$ is at most
\begin{align*}
  h & \left(1-\prod_{p\leq J}\left(1-\frac{1}{p^2}\right)\right)+2^{\pi(J)}          = h\left(1-\prod_{p\leq J}\left(1-\frac{1}{p^2}\right)+\frac{2^{\pi(J)}}{h}\right) \eqqcolon h\sigma'_0(h,J).
\end{align*}
We then count separately the integers divisible by $p^2$ for each prime $p>J$. We find that
\begin{equation}\label{eqn2.1}
  N(x,h) \leq h\sigma_0'(h,J) + \sum_{p > J} \left( \left\lfloor \frac{x+h}{p^2} \right\rfloor - \left\lfloor \frac{x}{p^2} \right\rfloor \right),
\end{equation}
where the sum on the right is over all primes greater than $J$. To bound the latter sum, we study separately the contributions of ``small'' and ``large'' primes $p$. We introduce a parameter $H$, which we will later choose as $H = mh$, with $m \ge 1$ of moderate size, and we use this parameter to split the sum in \eqref{eqn2.1} as follows:
\begin{equation}\label{eqn2.2}
  \left( \sum_{J < p \le H} + \sum_{p > H} \right) \left( \left\lfloor \frac{x+h}{p^2} \right\rfloor - \left\lfloor \frac{x}{p^2} \right\rfloor \right) \eqqcolon \Sigma_1 + \Sigma_2.
\end{equation}

The contribution of the small primes can be bounded easily. We have
\begin{align}\label{eq:moderatePrimes}
  \Sigma_1 & \leq \sum_{J < p \leq H} \left( \frac {h}{p^2} + 1 \right) \leq h \sum_{p > J} \frac {1}{p^2} + \pi(H) \notag \\
            & < h\left(\sigma_1 - \sum_{p \leq J} \frac{1}{p^2} \right) +\pi(H),
\end{align}
where $\sigma_1$, the sum of the reciprocals of all squares of primes, satisfies
\newcommand{\sigmaOnea}{0.4523}
\begin{equation}\label{eqn2.4}
  \sigma_1 < \sigmaOnea.
\end{equation}
We group the sum over primes up to $J$ appearing in \eqref{eq:moderatePrimes} with $\sigma'_0(h,J)$ to write
\begin{equation}
  \sigma_0(h,J) = 1-\prod_{p\leq J}\left(1-\frac{1}{p^2}\right)-\sum_{p \leq J} \frac{1}{p^2}+\frac{2^{\pi(J)}}{h}, \label{eq:sigma0}
\end{equation} so that we get
\begin{equation}
  N(h,x) \leq h\big(\sigma_0(h,J) + \sigma_1 \big) + \pi(H) + \Sigma_2. \label{eq:Sig013}
\end{equation}
The term $\pi(H)$ above can be bounded with the help of the following well-known result of Rosser and Schoenfeld \cite[(3.2)]{RS62}.

\begin{lemma}\label{lemRS}
  For any $x > 1$, one has
  \begin{equation}\label{eq:RosSchoen}
    \pi(x) < \frac {x}{\log x} \left( 1 + \frac {1.5}{\log x} \right).
  \end{equation}
\end{lemma}

Applying this lemma, we see that
\begin{equation}\label{eq:Sig2}
  \pi(H) < \sigma_2(h,m) h, \quad \sigma_2(h,m) \coloneqq \frac {m}{\log(mh)}\left( 1 + \frac {1.5}{\log (mh)} \right).
\end{equation}

The estimation of the sum $\Sigma_2$ occupies the remainder of the paper. We remark that primes $p > \sqrt{2x}$ do not contribute to that sum, since for such primes we have
\[ 0< \frac{x}{p^2}<\frac{x+h}{p^2}\leq\frac{2x}{p^2} < 1. \]
Moreover, if $p>h^{1/2}$, we get
\[ 0 \leq \left\lfloor \frac {x+h}{p^2} \right\rfloor - \left\lfloor \frac {x}{p^2} \right\rfloor\leq \frac {h}{p^2} + 1 < 2. \]
Thus, the finite sum $\Sigma_2$ counts the primes $p \in (H, \sqrt{2x}]$ for which
there exists an integer $m$ with
\[ \frac{x}{p^2} < m \leq \frac{x+h}{p^2}. \]
The latter inequality can be expressed in terms of the fractional part of $xp^{-2}$: it says that $\{ xp^{-2} \} > 1 - hp^{-2}$. Therefore,
\begin{equation}
  \Sigma_2 \leq \big| S(H, \sqrt{2x}) \big|, \label{eq:Sig3}
\end{equation}
where
\begin{equation} S(M,N) \coloneqq \left\{ u \in \mathbb Z : M< u \leq N, \; \gcd(u,2) = 1, \; 1 - \frac{h}{u^{2}} \le \left\{ \frac{x}{u^{2}} \right\} < 1 \right\}. \label{eq:Sdef} \end{equation}
We remark that while we no longer require the elements of $S(M,N)$ to be prime, we do restrict them to odd values so that the differences between any two elements of the set are even, a fact which will be useful later.

Thus, in view of \eqref{eq:Sig013}, \eqref{eq:Sig2}  and \eqref{eq:Sig3}, to prove any of our results, it will suffice to find a choice of $H$ such that
\begin{equation}\label{eqn2.9}
  \big| S(H, \sqrt{2x}) \big| \le h\sigma_3(h,m),
\end{equation}
for some bounded function $\sigma_3(h,m)$ such that
\begin{equation}\label{eqn2.10}
  \sigma_0(h,J) + \sigma_1 + \sigma_2(h,m) + \sigma_3(h,m) < 1 - \frac{1}{h}.
\end{equation}

In Section \ref{sec:6}, we establish inequalities of the form \eqref{eqn2.9} and optimize the choices of several parameters to ensure that the respective versions of \eqref{eqn2.10} hold. We conclude the present section with the statements of a couple of general-purpose lemmas, which we will use repeatedly in the remainder of the paper to obtain bounds on the spacing and cardinality of sets $S(M,N)$.

\subsection{Some general lemmas}

Our bounds on $|S(M, N)|$ are based on the simple idea that if the minimum distance between distinct elements of a set of integers $A$ is at least $d$, then
\begin{equation}\label{eqn2.11}
  |A \cap (M, N]| \le d^{-1}(N-M) + 1.
\end{equation}
In Sections \ref{sec:3}--\ref{sec:5}, we prove several results on the spacing between elements of sets 
\[S(M) := S(M,\lambda M),\] 
where $\lambda > 1$ is a constant. Those spacing estimates and inequality \eqref{eqn2.11} yield bounds on $|S(M)|$, which we leverage with the help of the next lemma.

\begin{lemma}\label{lemma-1}
  Suppose that $A_1, A_2, A_3, b_1, b_2$ are positive reals and $u$, $v$, $\lambda$ are real numbers with $0 < u < v < 1 < \lambda$. Assume that for all $M \in [x^u,x^v]$ the estimate
  \[ |S(M)| \leq A_1M^{b_1}+A_2M^{-b_2}+A_3\]
  holds. Then
  \[ |S(x^u,x^v)| \leq A_1'x^{b_1v}+A_2'x^{-b_2u}+A_3'\log x+A_3, \]
  where
  \[A_1'=\frac{A_1}{1-\lambda^{-b_1}},\quad A_2'=\frac{A_2}{1-\lambda^{-b_2}},\quad A_3'=A_3\cdot\frac{v-u}{\log{\lambda}}.\]
\end{lemma}

\begin{proof}
  This is standard: we cover the interval $(x^u,x^v]$ with intervals of the form $(M, \lambda M]$, apply the hypothesis to each of them, and sum the ensuing geometric progressions. The only (minimal) novelty in the present version is the explicit description of the coefficients $A_j'$ in terms of the $A_j$'s and the various parameters. The reader will find a detailed proof of a variant for $\lambda = 2$ in \cite[Lemma 1]{Fila90}.
\end{proof}

Some of our results also rely on the properties of divided differences. For a function $f: [a,b] \to \mathbb R$ and $s+1$ points $t_0, t_1, \dots, t_s \in [a,b]$, the {\em divided difference (of order $s$)}, $f[t_0, t_1, \dots, t_s]$, of $f$ at the given points is defined recursively: we set $f[t_0] = f(t_0)$ when $s=0$, and
\[ f[t_0, t_1, \dots, t_s] = \frac{f[t_1, \dots,t_s] - f[t_0, \dots, t_{s-1}]}{t_s-t_0} \]
when $s \ge 1$. Divided differences are a tool in numerical analysis that has a long and rich history, but here we are interested only in two of their elementary properties, which we summarize in the next lemma. The reader can find proofs of these properties in many texts on numerical analysis that discuss interpolation theory: e.g., \cite[Ch. 6]{IsKe94}.

\begin{lemma}\label{lem:divideddiff}
  Let $f: [a,b] \to \mathbb R$ be a function, let $t_0 < t_1 < \dots < t_s$ be distinct numbers in $[a,b]$, and let $f[t_0,t_1,\dots, t_s]$ denote the respective divided difference of $f$. Then
  \[ f[t_0,t_1,\dots,t_s] = \sum_{j=0}^s \frac {f(t_j)}{\prod\limits_{i \ne j}(t_j - t_i)}, \]
  where the product is over $i \in \{0,1,\dots,s\} \setminus \{ j \}$.
  Moreover, if $f$ has $s$ continuous derivatives on $[a,b]$, then there is a number $\xi \in (t_0, t_s)$ such that
  \[ f[t_0,t_1,\dots,t_s] = \frac {f^{(s)}(\xi)}{s!}. \]
\end{lemma}

\section{Basic Spacing Lemmas}
\label{sec:3}

Let $M$ be a large parameter, with $H \le M \le \sqrt{2x}$, and let $\lambda \in (1, 2]$ be a constant. In this section, we prove several lower bounds on the minimum distance between distinct elements of $S(M)$. As we pointed out in the introduction, the computational work in \cite{MOT21} allows us to assume that $x$ is large. Also, while in our proofs we will utilize several different choices for $h$ and $H$, we will always have $h \le H$ and $h \le 2x^{1/3}$. Thus, we assume in the remainder of the paper that
\begin{equation}\label{eqn3.0}
  x \ge e^{41}, \quad 1000 \le h \le 2x^{1/3}.
\end{equation}

\subsection{Spacing for pairs}

First, we show that two distinct elements of $S(M)$ cannot be ``too close'' to one another.

\begin{lemma}\label{lem:2}
  Suppose that $H \le M$. If $u$ and $u+a$ are distinct elements of $S(M)$, then
  \begin{equation}\label{eqn3.1}
    a > 0.4995x^{-1}M^{3}.
  \end{equation}
\end{lemma}

\begin{proof}
  Consider the function $f(u)=xu^{-2}$. If $u, u+a \in S(M)$, we have
  \begin{equation}\label{eq:defntheta}
    f(u)=n_1-\theta_1, \quad f(u+a)=n_2-\theta_2,
  \end{equation}
  with $n_1,n_2 \in \mathbb{Z}$,  $0<\theta_1,\theta_2<hM^{-2}$. So,
  \[ f(u+a)-f(u)=n-\theta, \qquad |\theta|<hM^{-2}. \]
By the mean-value theorem, there exists a number $\xi \in (u,u+a)$ such that
  \[ |f(u+a)-f(u)| = a|f'(\xi)| = \frac {2ax}{\xi^{3}} > \frac {2x}{(\lambda M)^3}.\]
If $n=0$, we have $|f(u+a) - f(u)| = |\theta| < hM^{-2}$, and we deduce that \[ 2\lambda^{-3}x < hM < 3x^{5/6}, \]
  which contradicts \eqref{eqn3.0}. Thus, we have $n \ne 0$, so $|n| \geq 1$.  We also get that
  \[ |\theta| < hM^{-2} \leq hH^{-2} \le H^{-1} \leq 0.001. \]
Hence, $|f(u+a)-f(u)| \geq 1 - |\theta| \geq 0.999$, and we obtain 
  \[ 0.999 \leq |f(u+a)-f(u)| = 2ax\xi^{-3} < 2axM^{-3}, \]
  from which \eqref{eqn3.1} follows.
\end{proof}

Applying \eqref{eqn2.11} to the result of the last lemma, we obtain the following bound on the size of $S(M)$.

\begin{corollary}\label{cor1}
  Under the hypotheses of Lemma \ref{lem:2}, we have
  \[ |S(M)| \le 0.4995^{-1}(\lambda - 1)xM^{-2} + 1. \]
\end{corollary}

\subsection{Spacing for triples}

Next, we consider any three distinct elements $u, u+a, u+b$ of $S(M)$, with $0 < a < b$, and obtain lower bounds on $b$.

\begin{lemma}\label{lemma4}
  Let $\lambda \le 1.2$, $m \ge 1.5$, and suppose that $mh = H \le M$. If $0 < a < b$ and $u,u+a,u+b$ are elements of $S(M)$, then
  \begin{equation}
    b \ge 1.3860x^{-1/3}M^{4/3}. \label{eqn3.11a}
  \end{equation}
\end{lemma}

\begin{proof}
  Suppose first that $b \le 0.004M$. Write $u_1=u$, $u_2=u+a$, and $u_3=u+b$, and let $n_1, n_2, n_3 \in \mathbb Z$ be such that
  \[ f(u_i)=n_i-\theta_i, \quad 0< \theta_i < hM^{-2} \qquad (i = 1,2,3). \]
  We consider the second divided difference $f[u_1,u_2,u_3]$. By Lemma \ref{lem:divideddiff},
  \begin{align*} f[u_1,u_2,u_3] & = \frac{f(u_1)(u_3-u_2)+f(u_2)(u_1-u_3)+f(u_3)(u_2-u_1)}{(u_2-u_1)(u_3-u_1)(u_3-u_2)}                 \\
                   & =\frac{(n_1-\theta_1)(b-a)-(n_2-\theta_2)b+(n_3-\theta_3)a}{ab(b-a)}\ \eqqcolon \ \frac{n-\theta}{V},
  \end{align*}
  where
  \[n=(b-a)n_1-bn_2+an_3\text{\qquad and \qquad}\theta=(b-a)\theta_1-b\theta_2+a\theta_3. \]
  In particular, since $\theta_i > 0$, we have
  \[ -bhM^{-2} < -b\theta_2 < \theta < (b-a)\theta_1 + a\theta_3 < bhM^{-2}. \]
  Moreover, since $u$, $u+a$ and $u+b$ are all odd (see \eqref{eq:Sdef}) we know that $a$ and $b$ are both even, so $n$ must be as well. 

  We will show that $n \neq 0$. Suppose that $n=0$. Then
  \[ |f[u_1,u_2,u_3]| = \frac {|\theta|}{V} < \frac{bhM^{-2}}{ab(b-a)} = \frac{h}{a(b-a)M^2} < \frac{hx}{0.999M^5}, \]
  after an appeal to \eqref{eqn3.1} and the bound $b-a \ge 2$. However, using Lemma \ref{lem:divideddiff}, we also get that
  \[ |f[u_1,u_2,u_3]| = \frac{|f''(\xi)|}{2!} = \frac{3x}{\xi^4} \geq \frac{3x}{(\lambda M)^4}. \]
  Thus,
  \[ \frac{3x}{(\lambda M)^4} <  \frac{hx}{0.999M^5} < \frac{1.002hx}{HM^4}, \]
  which contradicts the hypotheses of the lemma.

Having proved that $n \neq 0$ and using that it is even, we find that $|n|\geq 2$.
Hence,
  \begin{equation} \label{eqn3.10}
    |f[u_1,u_2,u_3]| = \frac{|n-\theta|}{V} \geq \frac{2-|\theta|}{ab(b-a)} > \frac{1.997}{ab(b-a)},
  \end{equation}
  since
  \[ |\theta|< bhM^{-2}< 0.004hH^{-1} \leq \frac 1{250m} \le \frac 1{375}. \]
  On the other hand, by Lemma \ref{lem:divideddiff},
  \begin{equation} \label{eqn3.11}
    |f[u_1,u_2,u_3]| = \frac{3x}{\xi^4} \leq \frac {3x}{M^4}. 
  \end{equation}
  From \eqref{eqn3.10}, \eqref{eqn3.11}, and the elementary inequality $a(b-a) \le \frac 14b^2$, we deduce that
  \begin{equation} \label{eqn3.12}
    \frac {3b^3x}4 \geq 3ab(b-a)x > 1.997M^4,
  \end{equation}
  and the conclusion of the lemma follows in the case $b \le 0.004M$.

  Finally, when $b > 0.004M$, we have
  \[ b^3 > (0.004M)^3  > \frac 83x^{-1}M^4, \]
  by the assumptions that $M \le \sqrt{2x}$ and $x\geq e^{41}$.
\end{proof}

Note that the expression on the right side of \eqref{eqn3.11a} is a lower bound for the minimum distances between successive elements of the set $S_1(M)$ containing every other element of $S(M)$. Since $|S(M)| \le 2|S_1(M)|$, this observation and \eqref{eqn2.11} yields the following corollary.

\begin{corollary}\label{cor3}
  Under the hypotheses of Lemma \ref{lemma4}, we have
  \[ |S(M)| \le 1.4430(\lambda-1)x^{1/3}M^{-1/3} + 2. \]
\end{corollary}

\section{Spacing for Pairs of Pairs}

In this section, we study a special family of quadruples $u$, $u+a$, $u+b$, $u+a+b$ of elements of $S(M)$. The special form of the spacing between the four numbers allows us to obtain bounds on $b$ that are stronger than those for general quadruples in $S(M)$; in the next section, we will average these bounds over $b$. In the next lemma, we use the third-order divided difference of $f(u) = xu^{-2}$ for the points $u$, $u+a$, $u+b$, and $u+a+b$ to bound $b$ from below.

\begin{lemma}
  Let $\lambda \le 1.05$, $m \ge 5$, and suppose that $mh \le H \le M$. If $0 < a < 2a \le b$ and $u,u+a,u+b,u+a+b$ are elements of $S(M)$, then
  \begin{equation}\label{eqn4.1}
    ab^3 \geq 0.6600x^{-1}M^5.
  \end{equation}
\end{lemma}

\begin{proof}
  Consider points $u_1=u$, $u_2=u+a$, $u_3=u+b$, and $u_4=u+a+b$ in $S(M)$. Recall that by the definition of the set $S(M)$, there exist integers $n_1, \dots, n_4$ and reals $\theta_1, \dots, \theta_4$ such that
  \begin{equation}\label{eqn4.2}
    f(u_i) = n_i-\theta_i, \quad 0<\theta_i<hM^{-2} \qquad (1 \le i \le 4).
  \end{equation}
  We consider the divided difference $f[u_1, \dots, u_4]$.

  Due to the special configuration of the distances between the four points, the formula in Lemma \ref{lem:divideddiff} simplifies to
  \begin{align*}
    f[u_1,u_2,u_3,u_4] & = \frac{f(u_4)-f(u_1)}{ab(a+b)}-\frac{f(u_3)-f(u_2)}{ab(b-a)} =: \frac {n - \theta}V,
  \end{align*}
  where $V = ab(a+b)(b-a)$ and
  \begin{align*}
    n      & = (b-a)(n_4 - n_1) - (a+b)(n_3 - n_2),                \\
    \theta & =(b-a)(\theta_4-\theta_1) - (a+b)(\theta_3-\theta_2).
  \end{align*}
  We remark that $n$ is an even integer and $|\theta| < 2bhM^{-2}$.

  We will show that $n \neq 0$. Suppose that $n = 0$. Then
  \[ |f[u_1, \dots, u_4]| = \frac {|-\theta|}{V} \leq \frac {2bhM^{-2}} {ab(a+b)(b-a)}. \]
  Recalling \eqref{eqn3.12}, we deduce that
  \[ |f[u_1, \dots, u_4]| < \frac {2hM^{-2}} {ab(b-a)} \leq \frac {6hx} {1.997M^6}. \]
  However, Lemma \ref{lem:divideddiff} gives
  \[ |f[u_1, \dots, u_4]| = \frac{|f^{(3)}(\xi)|}{3!} = \frac {4x}{\xi^5} \geq \frac{4x}{(\lambda M)^5}, \]
  for some $\xi \in (M, \lambda M]$. We combine these upper and lower bounds to get
  \[ \frac{4x}{(\lambda M)^5} < \frac {6hx} {1.997M^6} < \frac {3.005hx}{HM^5}, \]
  which contradicts the assumptions of the lemma.

  Since $n$ is even and nonzero, we can now use that $|n| \ge 2$ combined with the observation $b^2 - a^2 \ge 0.75b^2$ to obtain
  \begin{equation} \label{eqn4.3}
    |f[u_1, \dots, u_4]| = \frac{|n-\theta|}{V} \geq \frac{2-|\theta|}{ab(b^2-a^2)} \geq \frac{1.98}{0.75ab^3},
  \end{equation}
  since
  \[ |\theta|< 2bhM^{-2}< 2(\lambda-1)hH^{-1} \leq \frac 1{10m} \le \frac 1{50}. \]
  On the other hand, by Lemma \ref{lem:divideddiff},
  \begin{equation} \label{eqn4.4}
    |f[u_1, \dots, u_4]| = \frac{4x}{\xi^5} \leq \frac {4x}{M^5}. \end{equation}
  The lemma follows from \eqref{eqn4.3} and \eqref{eqn4.4}.
\end{proof}

Our next result is of a somewhat different nature from the spacing lemmas established hitherto. In this lemma, instead of proving that the distance $b$ between the two pairs exceeds some lower bound in terms of $x, M$, and possibly, $a$, we establish a kind of a dichotomy for $b$: either $b \ge B_1$ for some lower bound $B_1$, or $b \le B_2$, with $B_2$ significantly smaller than $B_1$. 

\begin{lemma}\label{lem:9}
  Let $\lambda \le 1.05$, $m \ge 5$, and suppose that $mh \le H \le M$. If $0 < a < 2a \le b$ and $u,u+a,u+b,u+a+b$ are elements of $S(M)$, then exactly one of the conditions
  \begin{equation}\label{eqn4.5}
    a^3b < \lambda^6 hx^{-1}M^4,
  \end{equation}
  or
  \begin{equation}\label{eqn4.6}
    a^3b > (0.5 - \lambda m^{-1})x^{-1}M^5,
  \end{equation}
  must hold.
\end{lemma}

\begin{proof}
  We start from the algebraic identity
  \[ \frac{2u+3a}{(u+a)^2} - \frac{2u-a}{u^2} = \frac {a^3}{u^2(u+a)^2}. \]
  Since $u, u+a \in S(M)$, we can use this identity and \eqref{eqn4.2} to get that
  \begin{gather}\label{eqn4.7}
    \frac{a^3x}{u^2(u+a)^2} = \frac{(2u+3a)x}{(u+a)^2} - \frac{(2u-a)x}{u^2} = n' + \theta'
  \end{gather}
  where $n' = (2u+3a)n_2-(2u-a)n_1$ is an even integer and
  \begin{align*}
    |\theta'| & = |\theta_1(2u-a)-\theta_2(2u+3a)| \leq 2u|\theta_1-\theta_2|+a(\theta_1+3\theta_2) < (2u + 4a)hM^{-2}.
  \end{align*}
  Combining \eqref{eqn4.7} with the analogous identity for the pair $u+b, u+a+b$, we find that
  \begin{equation}\label{eqn4.8}
    \frac{a^3x}{u^2(u+a)^2} - \frac{a^3x}{(u+b)^2(u+a+b)^2} = n + \theta,
  \end{equation}
  where $n \in \mathbb Z$ is even and
  \[ |\theta| < (4u + 2b + 8a) hM^{-2} \le 4(u+a+b)hM^{-2} \le 4\lambda hM^{-1} \le 4\lambda m^{-1}. \]

  Next we observe, by the mean-value theorem, there is a $\xi \in (u,u+b)$ such that
  \begin{align*}
    \frac{a^3x}{u^2(u+a)^2} - \frac{a^3x}{(u+b)^2(u+a+b)^2}
     & = \frac{2a^3bx(2\xi+a)}{\xi^3(\xi+a)^3}.
  \end{align*}
  This expression is bounded above by
  \begin{align}\label{eqn4.9}
    \frac{2a^3bx(2\xi+a)}{\xi^3(\xi+a)^3} < \frac{4a^3bx}{\xi^3(\xi + a)^2} < 4a^3bxM^{-5},
  \end{align}
  and bounded below by
  \begin{align}\label{eqn4.10}
    \frac{2a^3bx(2\xi+a)}{\xi^3(\xi+a)^3} > \frac{4a^3bx}{\xi^2(\xi + a)^3} > 4a^3bx(\lambda M)^{-5}.
  \end{align}
  When $a^3b \le (0.5 - \lambda m^{-1})x^{-1}M^5$, \eqref{eqn4.8}, \eqref{eqn4.9}, and the bound on $|\theta|$ yield
  \[n- 4\lambda m^{-1} < n+\theta < 4a^3bxM^{-5} \leq 2 - 4\lambda m^{-1}, \]
  and hence, $n<2$. On the other hand, if $a^3b \ge \lambda^6 hx^{-1}M^4$, we deduce from \eqref{eqn4.8} and \eqref{eqn4.10} that
  \[ 4\lambda hM^{-1} < 4a^3bx(\lambda M)^{-5}  < n+\theta < n+4\lambda hM^{-1}, \]
  so in this case $n>0$. Since $n$ is an even integer, it can satisfy only one of the conditions $n>0$ and $n<2$; therefore, at least one of \eqref{eqn4.5} or \eqref{eqn4.6} must hold. This completes the proof, since under the hyptheses of the lemma, the lower bound in \eqref{eqn4.6} exceeds the upper bound in \eqref{eqn4.5} at least by a constant factor.
\end{proof}

\section{The Main Bounds on $|S(M)|$}
\label{sec:5}

Let
\begin{equation}\label{eqn5.1}
  A = 1.3860x^{-1/3}M^{4/3}.
\end{equation}
In Section \ref{sec:3}, we proved that $b \ge A$ whenever $u, u+a, u+b$ are distinct elements of $S(M)$. Therefore, if $u_0, u_1, \dots, u_s$ are the elements of $S(M)$, listed in increasing order, the set $S_1(M) = \{u_0,u_2,u_4, \dots \}$ has no gaps $<A$ and satisfies
\begin{equation}\label{eqn5.2}
  |S(M)| \le 2|S_1(M)|.
\end{equation}

In this section, we use \eqref{eqn5.2} and the lemmas in the last section to prove the following result.

\begin{proposition}
  \label{prop1}
  Suppose $h=11x^{1/5}\log x$, let $\lambda = 1.045$ and $x \ge e^{116}$, and suppose that $5.5h \le M \le x^{2/5}$. Then
  \begin{equation}\label{eqn5.3}
    |S(M)| \le h (\sigma_3(M) + \sigma_4(M)),
  \end{equation}
  where
  \begin{equation}\label{eqn5.4}
    \sigma_3(M) = \big( 0.5298x^{1/5} + 0.3400x^{-1/5}M \big)h^{-1} + 0.0308 x^{1/15}M^{-1/3},
  \end{equation}
  and
  \begin{equation}
    \sigma_4(M) = \begin{cases} 1.2105x^{-2/3}M^{7/3} & \text{if } M \le 5x^{1/4}, \\
      1.4182x^{-1/9}M^{1/9} & \text{if } M > 5x^{1/4}.
    \end{cases} \label{eqn5.5}
  \end{equation}
\end{proposition}

\begin{remark} Notice that when $x$ is relatively small, the condition $M \le 5x^{1/4}$ in Proposition \ref{prop1} is  impossible, and so only the second condition will be used in the range of ``small'' values of $x$.
\end{remark}

The proof of this proposition uses the set
\[ T(M;a) = \left\{ u : u, u+a \text{ are consecutive elements of } S_1(M) \right\} \]
to bound $|S_1(M)|$. The starting point is the elementary identity
\begin{equation}\label{eqn5.11}
  |S_1(M)| = 1 + \sum_{a=1}^\infty |T(M;a)| = 1 + \sum_{a \ge A} |T(M;a)|,
\end{equation}
which is a direct consequence of the definition of $T(M;a)$. Further, for any $B \ge A$, we have
\[ \sum_{a \ge B} a|T(M;a)| \le \sum_{a \ge A} a|T(M;a)| \le (\lambda - 1)M + 1, \]
so
\[ \sum_{a \ge B} |T(M;a)| \le (\lambda - 1)MB^{-1} + B^{-1}. \]

Applying this inequality to the right side of \eqref{eqn5.11}, we find, for any parameter $B \ge 2$, that
\begin{equation}\label{eqn5.12}
  |S_1(M)| \le 1.5 + (\lambda - 1)MB^{-1} + \sum_{A \le a < B} |T(M;a)|.
\end{equation}

\subsection{Proof of Proposition \ref{prop1}}
We recall the quantity $A$ defined in \eqref{eqn5.1}, and we select
\begin{equation}
  B = \delta x^{-1/5}M, \quad \delta = 0.17, \label{eqn5.13}
\end{equation}
in the imminent application of \eqref{eqn5.12}. We fix an integer $a$, with $A \le a \le B$. If $u_0, u_1, \dots, u_t$ are the elements of $T(M;a)$, listed in increasing order, the set $T_1(M) = \{ u_0, u_2, u_4, \dots \}$ contains only elements of $T(M;a)$ such that if $u,u+b \in T_1(M)$, then $b \ge 2a$. Clearly, $|T(M;a)| \le 2|T_1(M)|$.

Let $I$ be a subinterval of $(M, \lambda M]$ of length \[|I| = (0.5 -\lambda m^{-1}) a^{-3}x^{-1}M^5,\]  and let $u,u+b$ be two elements of $T_1(M) \cap I$. Since $b \ge 2a$, we can apply Lemma~\ref{lem:9} to show that $b$ must satisfy \eqref{eqn4.5}. Taking $u$ and $u+b$ to be the smallest and largest elements of $T_1(M) \cap I$ respectively, we can use this bound on $b$ to deduce that the set $T_1(M) \cap I$ is contained in an interval of length  $\leq \lambda^6 a^{-3}hx^{-1}M^4$. Furthermore, by \eqref{eqn4.1}, we have that
\[ b \geq 0.8706a^{-1/3}x^{-1/3}M^{5/3}. \]
Combining these two observations we find that 
\begin{equation} 
|T_1(M) \cap I| \leq \frac {\lambda^6 a^{-3}hx^{-1}M^4}{0.8706a^{-1/3}x^{-1/3}M^{5/3}}+1 < 1.4959 a^{-8/3}hx^{-2/3}M^{7/3} + 1. 
\end{equation}
Since we need at most
\begin{equation}  \frac{(\lambda-1)M}{(0.5 -\lambda m^{-1}) a^{-3}x^{-1}M^5}+1 < 0.1452 a^3xM^{-4}+1 \end{equation}
intervals of length $|I|$ to cover $(M,\lambda M]$, we conclude that
\begin{align*}
  |T_1(M)| & \leq \left( 0.1452 a^3xM^{-4}+1 \right) \left( 1.4959 a^{-8/3}hx^{-2/3}M^{7/3} + 1 \right) \notag       \\
          & < 0.2173 a^{1/3}hx^{1/3}M^{-5/3}+ 0.1452 a^3xM^{-4}+1.4959 a^{-8/3}hx^{-2/3}M^{7/3}+1.
\end{align*}
Thus,
\begin{align}
  |T(M;a)| < 0.4346 a^{1/3}hx^{1/3}M^{-5/3}+ 0.2904 a^3xM^{-4} \notag\\
 + 2.9918 a^{-8/3}hx^{-2/3}M^{7/3} + 2. \label{eq:T2Upbd}
\end{align}

Next, we use \eqref{eq:T2Upbd} to bound the right side of \eqref{eqn5.12}.  With our choice of parameters, \eqref{eqn5.12} gives
\begin{equation}
  |S_1(M)| \le 1.5 + 0.045\delta^{-1} x^{1/5} + \sum_{A \le a < B} |T(M;a)|. \label{eqn5.15}
\end{equation}
Thus, we need to sum each of the four terms on the right side of \eqref{eq:T2Upbd} over $a \in [A, B)$. Recalling the inequality
\[ \sum_{k \le K} k^s < \frac{(K+1)^{s+1}}{s+1} \qquad (s > 0), \]
and noting that $B =\delta M x^{-1/5} \ge 5.5h\delta x^{-1/5} > 10.285\log x > 1193$, we find that
\begin{equation}
\sum_{\substack{2 \le a \le B\\ a \text{ even}}} a^s < \frac{(B + 2)^{s+1}}{2(s+1)} < \frac{(1.002B)^{s+1}}{2(s+1)}. \label{eqn5.15a}
\end{equation}
Hence,
\begin{align}
  0.4346 hx^{1/3}M^{-5/3} \sum_{\substack{A \le a < B\\ a \text{ even}}} a^{1/3}   
  &< \frac{0.4346 \cdot (1.002B)^{4/3}}{8/3} hx^{1/3}M^{-5/3} \nonumber \\
  & < 0.0154 hx^{1/15} M^{-1/3},  \label{eqn5.17b}
\end{align}
and
\begin{align}
  0.2904 xM^{-4} \sum_{\substack{A \le a < B                                                                                             \\ a \text{ even}}} a^3 <
  \frac{0.2904 \cdot (1.002B)^4}{8}xM^{-4} < 0.00004 x^{1/5} . \label{eqn5.16}
\end{align}
Combining \eqref{eqn5.13}, \eqref{eq:T2Upbd}, \eqref{eqn5.15}, \eqref{eqn5.17b}, and \eqref{eqn5.16}, we conclude that
\begin{equation}
  |S_1(M)| \le h\sigma_3'(M) + 2.9918 hx^{-2/3}M^{7/3}\sum_{\substack{A \le a < B\\ a \text{ even}}} a^{-8/3}, \label{eqn5.18}
\end{equation}
where
\begin{align}
  \sigma_3'(M) = \big( 0.2649x^{1/5} + 0.17x^{-1/5}M \big)h^{-1} + 0.0154x^{1/15}M^{-1/3}.  \label{eqn5.19}
\end{align}

We estimate the sum on the right side of \eqref{eqn5.18} in different ways, depending on the size of $M$. When $M \le 5x^{1/4}$, we use that
\begin{equation}  \sum_{\substack{A \le a < B\\ a \text{ even}}} a^{-8/3} < \frac{\zeta(8/3)}{2^{8/3}} < 0.2023. \label{eqn5.20}
\end{equation}
On the other hand, when $M > 5x^{1/4}$, we have $A > 1.386 \cdot 5^{4/3} > 11.8501$, so
\begin{equation}
  \sum_{\substack{A \le a < B\\ a \text{ even}}} a^{-8/3} < \frac 3{5\cdot 2^{8/3}} \left(\frac{A}{2}-1\right)^{-5/3} < 0.4083A^{-5/3} < 0.237x^{5/9}M^{-20/9}. \label{eqn5.21}
\end{equation}

The proposition follows from \eqref{eqn5.2} and \eqref{eqn5.18}--\eqref{eqn5.21}.
\qed

\section{Proof of Theorem \ref{thm2free1}}
\label{sec:6}

The proof of the theorem uses different approaches for different values of $x$. As we stated in the introduction, the work in \cite{MOT21} establishes our result (and much more) for $x \le e^{41}$. In Section \ref{sec:6.1}, we focus on large $x$ and show that for $x \ge e^{116}$, Theorem~\ref{thm2free1} follows from Proposition \ref{prop1}. To complete the proof, in Section \ref{sec:6.2}, we prove two asymptotically weaker variants, which are, however, stronger than the theorem for small $x$. Those alternative results establish Theorem \ref{thm2free1} in the intermediate range $e^{41} \le x \le e^{116}$.

\subsection{Large $x$.}
\label{sec:6.1}

Let $x \ge e^{116}$ and set $H = 5.5h$ in \eqref{eqn2.2} and \eqref{eqn2.9}. First, we use Proposition~\ref{prop1} and Lemma \ref{lemma-1} to bound $\left|S\left(H, x^{2/5}\right)\right|$.

Suppose first that $H \leq 5x^{1/4}$, (in this case we can assume $x \ge e^{150}$) we split $S(H, x^{2/5})$ in two pieces to account for the different cases in \eqref{eqn5.5}. When we apply Lemma \ref{lemma-1} to the bound \eqref{eqn5.3} for $M \in [H,5x^{1/4}]$, we find that
\begin{align*}
\left|S\left(H,5x^{1/4}\right)\right|  &<  h\left(  \frac {1.2105 \cdot 5^{7/3}x^{-1/12}}{1-1.045^{-7/3}} + \frac {0.0308x^{1/15}H^{-1/3}}{1-1.045^{-1/3}} \right) + \frac{1.7x^{1/20}}{1-1.045^{-1}} \\
& \quad \ + 0.5298x^{1/5} \left(\frac{\log \big( 5x^{1/4}/H \big)}{\log(1.045)}+1\right) \\
&\mkern-36mu \mkern-18mu < 0.1034h + 0.0001x^{1/5} + 0.5298x^{1/5} \left(\frac{\log x}{20\log(1.045)} - \frac{\log(12.1\log x)}{\log(1.045)}+1\right) \\
& \mkern-36mu \mkern-18mu < 0.1034h + 0.6019 x^{1/5}\log x  - 89.7886x^{1/5} < 0.1582h.
\end{align*}
Similarly, when we apply Lemma \ref{lemma-1} to \eqref{eqn5.3} for $M \in [5x^{1/4}, x^{2/5}]$, we get
\begin{align*}
\left|S\left(5x^{1/4}, x^{2/5}\right)\right| & <  h\left(  \frac {1.4182x^{-1/15}}{1-1.045^{-1/9}} + \frac {0.0308\cdot 5^{-1/3} x^{-1/60}}{1-1.045^{-1/3}} \right) + \frac{{0.34x^{1/5}}}{1-1.045^{-1}} \\
& \quad \ + 0.5298x^{1/5} \left(\frac{3\log x}{20\log(1.045)} -\frac{\log 5}{\log(1.045)}+1\right) \\
& < 0.1148h + 1.8055x^{1/5}\log x -10.9463x^{1/5} < 0.2790h.
\end{align*}
Hence,
\begin{equation}
  \left|S\left(H, x^{2/5}\right)\right| < 0.1582h +0.2790h = 0.4372h. \label{eqn6.1}
\end{equation}

Next, we consider the case $H > 5x^{1/4}$ (which implies that $x \le e^{151}$).  In this case we need only consider the latter case of Proposition \ref{prop1} for $M$ in the full range $(H,x^{2/5}]$.  Applying Lemma \ref{lemma-1} in this situation gives
\begin{align}
\left|S\left(H, x^{2/5}\right)\right| & <  h\left(  \frac {1.4182x^{-1/15}}{1-1.045^{-1/9}} + \frac {0.0308 x^{1/15}H^{-1/3}}{1-1.045^{-1/3}} \right) + \frac{0.34x^{1/5}}{1-1.045^{-1}} \notag \\
& \quad \ + 0.5298x^{1/5} \left(\frac{\log x}{5\log(1.045)} -\frac{\log(60.5\log x)}{\log(1.045)}+1\right) \notag \\
& < 0.2378h + 2.4073x^{1/5}\log x - 98.1700x^{1/5} < 0.3976h, \label{eqn6.2}
\end{align}
on noting that $98.170x^{1/5} > 0.059h$ when $x \le e^{151}$.

To complete the estimation of $\big|S(H, \sqrt{2x})\big|$, we apply Lemma \ref{lemma-1} to the bound in Corollary~\ref{cor1} for $M \in [x^{2/5}, \sqrt{2x}]$. This yields
\begin{equation}
  \big|S(x^{2/5}, \sqrt{2x})\big| <   \frac {0.0901x^{1/5}}{1-1.045^{-2}} + \frac{\log x}{10\log(1.045)} + \frac{0.5\log 2}{\log(1.045)}+1 < 0.0009h. \label{eqn6.3}
\end{equation}

Together, \eqref{eqn6.1}--\eqref{eqn6.3} establish \eqref{eqn2.9} with
\[\sigma_3 = \begin{cases}
    0.4381 &\text{if } H \leq 5x^{1/4}, \\
    0.3985 &\text{if } H > 5x^{1/4},
  \end{cases}\]
for all $x \ge e^{116}$. Taking $J=120$ in \eqref{eq:sigma0}, we have $\sigma_0(h,120) \leq -0.0595$ in the same range.
Furthermore, for all $x \ge e^{116}$, we have $\sigma_2(h, 5.5) < 0.1797$, and for $x \ge e^{150}$, we have $\sigma_2(h, 5.5) < 0.1461$. Thus,
\[ \sigma_0(h,120)+\sigma_1 + \sigma_2(h,5.5) + \sigma_3(h,5.5) < \begin{cases}
    0.9770 &\text{if } H \leq 5x^{1/4}, \\
    0.9710 &\text{if }  H > 5x^{1/4}, \\
  \end{cases}\]
which establishes \eqref{eqn2.10}, and therefore the theorem, for $x \ge e^{116}$.

\subsection{Intermediate $x$.}
\label{sec:6.2}

Suppose that $x \ge e^{41}$. We consider $h = 5x^{1/4}$ and choose $\lambda = 1.025$, $J=19$ and $H = 1.75h$. With these choices, we apply Lemma \ref{lemma-1} to the result of Corollary \ref{cor3} to obtain
\begin{align*}
  \big|S(H, \sqrt{2x})\big| & <   \frac{1.4430 \cdot (0.025)x^{1/3}H^{-1/3}}{1-1.025^{-1/3}} +2\left(\frac{\log\left(\sqrt{2x}\right)-\log H}{\log(1.025)} + 1 \right)                              \\
&< 0.4272h + \frac{\log x}{2\log(1.025)} + \frac {\log 2 - 2\log (8.75/1.025)}{\log(1.025)} < 0.4331h.
\end{align*} 
That is, \eqref{eqn2.10} holds with $\sigma_3(h,1.75) = 0.4331$. Moreover, when $h = 5x^{1/4}$ and $x \ge e^{41}$, we have
\begin{align*}
  \sigma_0(h,19) \le -0.0543, \quad \sigma_2(h,1.75) \le 0.158.
\end{align*}
Thus, when $h = 5x^{1/4}$ and $x \ge e^{41}$, we have
\[ \sigma_0(h,19)+ \sigma_1 + \sigma_2(h,1.75) + \sigma_3(h,1.75) < 0.9891. \]
Together with the computations of \cite{MOT21}, this proves the following result.

\begin{proposition}\label{prop3}
  For any $x \ge 2$, the interval $(x, x + 5x^{1/4}]$ contains a squarefree integer.
\end{proposition}

Moreover, an identical calculation for $x \geq e^{109}$ with  $h = 3.8x^{1/4}$, $H = 4.5h$, $\lambda = 1.0001$, and $J=100$ yields
\begin{align*}
  &\sigma_0(h,100)+ \sigma_1 + \sigma_2(h,4.5) + \sigma_3(h,4.5) \\
  &< -0.0594+0.4523+0.1571+0.4423 = 0.9924,
\end{align*}
which yields the following alternative.

\begin{proposition}\label{prop4}
  For any $x \geq e^{109}$, the interval $(x, x + 3.8x^{1/4}]$ contains a squarefree integer.
\end{proposition}

Since $5x^{1/4} \le 11x^{1/5}\log x$ for $x \le e^{109.7}$, Proposition \ref{prop3} implies Theorem \ref{thm2free1} for $x \le e^{109}$. Finally, since $3.8x^{1/4} \le 11x^{1/5}\log x$ for $x \le e^{116.3}$, Proposition \ref{prop4} establishes Theorem \ref{thm2free1} when $e^{109} \le x \le e^{116}$. This completes the proof of the theorem.

\section{Asymptotic Results and Final Comments}
\label{sec:8}

We conclude by noting a few of the explicit bounds that can be obtained by these methods if one no longer requires the bounds to be admissible for all values of $x \geq 2$, allowing instead results valid for sufficiently large values of $x$. 

Some of the possible results that can be obtained by tweaking the parameters used in the proof of Theorem \ref{thm2free1} are given in the statement of Theorem \ref{thm2free2}. To prove any of those results, we reset the parameters $m, J, \lambda, \delta$ that appear in the proofs of Proposition \ref{prop1} and Theorem \ref{thm2free1} and then update the various constants. (When $x$ is as large as in Theorem \ref{thm2free2}, the inequality $H \le 5x^{1/4}$ always holds, so only the first case in the proof of Theorem \ref{thm2free1} can occur.) To establish the claims of Theorem \ref{thm2free2}, we always select $J = 100$, $\lambda = 1.02$, and  $m = \sqrt{\log x_0}$, where $x_0$ is the lower bound on $x$ in each result; we only vary the choice of $\delta$. For example, when $h = 5x^{1/5}\log x$, $x \ge e^{400}$ (hence, $m = 20$), and $\delta = 0.3$, we have
\[ \sigma_0(h,100) + \sigma_1 + \sigma_2(h,m) + \sigma_3(h,m) < 0.9811. \]
For $h=2x^{1/5}\log x$ and $x \ge e^{1800}$,  the choice $\delta = 0.6$ yields an upper bound of $0.9857$; and for $h=x^{1/5}\log x$ and $x \ge e^{500\,000}$, $\delta = 0.87$ gives a bound of $0.9981$. 

\begin{remark}
Looking back at the proofs of our theorems, one can see that the value of $h$ in our theorems is of the form $h(x) = cx^{1/5}\log x$, with $c$ an upper bound for a rather complicated bounded function $C(x; m, J, \lambda, \delta)$, which is decreasing in the variable $x$. Once $x$ is sufficiently large, the decay in $x$ appears to overwhelm the effect of the other parameters. On the other hand, to claim a specific value of $c$ for all $x \ge x_0$, one generally needs to find acceptable choice of the other parameters to ensure that \eqref{eqn2.10} holds. It seems that if one were to make the function $C(x; m, J, \lambda, \delta)$ fully explicit, one may even be able to identify a four-dimensional neighborhood of the chosen values of $m, J, \lambda, \delta$ such that all the choices of the parameters in that neighborhood are acceptable. 
\end{remark}

\subsection*{Acknowledgments} This work is the result of an REU project that took place on the campus of Towson University during the summer of 2022, with the financial support of the National Science Foundation under grants DMS-2136890 and DMS-2149865. The authors also acknowledge financial support from the Fisher College of Science and Mathematics and TU's Mathematics Department. We also thank M. Filaseta, O. Trifonov, and T. Trudgian for commenting on early drafts of this work and to the anonymous referee for their careful reading of the manuscript and for making an important suggestion regarding the overall organization of the paper.

\bibliographystyle{amsplain}

\end{document}